\documentclass[psamsfonts]{amsproc}

\newcommand{\RR}{\mathbb{R}}
\newcommand{\ZZ}{\mathbb{Z}}
\newcommand{\bv}{\mathbf{v}}
\newcommand{\bw}{\mathbf{w}}
\newcommand{\bfone}{\mathbf{1}}

\numberwithin{equation}{section}
\newtheorem{theorem}[equation]{Theorem}
\newtheorem{lemma}[equation]{Lemma}
\newtheorem{cor}[equation]{Corollary}

\DeclareMathOperator{\PSL}{PSL}

\begin{document}

\title[Product-free subsets of groups, then and now]{Product-free subsets of groups, then and now}
\author{Kiran S. Kedlaya}
\address{Department of Mathematics, Massachusetts Institute of Technology,
77 Massachu\-setts Avenue, Cambridge, MA 02139}
\email{kedlaya@mit.edu}
\urladdr{http://math.mit.edu/~kedlaya/}
\dedicatory{Dedicated to Joe Gallian on his 65th birthday and the 30th
anniversary of the Duluth REU}
\thanks{The author was supported by 
NSF CAREER grant DMS-0545904 and a Sloan Research Fellowship.}
\date{November 7, 2007}
\subjclass[2000]{Primary 20D60; secondary 20P05.}

\maketitle

\section{Introduction}

Let $G$ be a group. A subset $S$ of $G$ is \emph{product-free} if
there do not exist $a,b,c \in S$ (not necessarily distinct\footnote{In some
sources, one does require $a \neq b$. For instance, as noted in
\cite{guiduci-hart}, I mistakenly assumed this in 
\cite[Theorem~3]{kedlaya-amm}.})
such that $ab=c$. 

One can ask about the existence of large product-free subsets for various
groups, such as the groups of integers (see next section), or compact
topological groups (as suggested in \cite{kedlaya-amm}). For the rest of this
paper, however, I will require $G$ to be a finite group of order
$n > 1$. Let $\alpha(G)$ denote the size of the largest product-free subset of 
$G$; put $\beta(G) = \alpha(G)/n$, so that $\beta(G)$ is the density of
the largest product-free subset. What can one say about 
$\alpha(G)$ or $\beta(G)$ as a function
of $G$, or as a function of $n$?
(Some of our answers will include an unspecified positive constant;
I will always call this constant $c$.)

The purpose of this paper is threefold. I first
review the history of this problem, up to and including
my involvement via Joe Gallian's REU (Research Experience for Undergraduates)
at the University of Minnesota, Duluth,
in 1994;
since I did this once already in \cite{kedlaya-amm}, I will be briefer
here. I then describe some very recent progress made by Gowers
\cite{gowers}. Finally, I speculate on the gap between the lower and upper
bounds, and revisit my 1994 argument to show that this gap cannot be
closed using Gowers's argument as given.

Note the usual convention that multiplication and inversion are
permitted to act on subsets of $G$, i.e., for $A,B \subseteq G$,
\[
AB = \{ab: a \in A, b \in B\}, \qquad A^{-1} = \{a^{-1}: a \in A\}.
\]

\section{Origins: the abelian case}

In the abelian case, product-free subsets are more customarily called
\emph{sum-free} subsets.
The first group in which such subsets were studied is the
group of integers $\ZZ$; the first reference I could find for 
this is Abbott and Moser \cite{abbott-moser},
who expanded upon Schur's theorem that the set
$\{1, \dots, \lfloor n! e \rfloor\}$ cannot be partitioned into
$n$ sum-free sets. This led naturally to considering sum-free
subsets of finite abelian groups, for which the following is easy.
\begin{theorem}
For $G$ abelian, $\beta(G) \geq \frac{2}{7}$.
\end{theorem}
\begin{proof}
For $G = \ZZ/p\ZZ$ with $p > 2$,
we have $\alpha(G) \geq \lfloor 
\frac{p+1}{3} \rfloor$ by taking 
\[
S = \left\{ \left\lfloor \frac{p+1}{3} 
\right\rfloor, \dots, 2 \left\lfloor \frac{p+1}{3}
\right\rfloor - 1 \right\}.
\]
Then apply the following lemma.
\end{proof}

\begin{lemma} \label{L:quotient}
For $G$ arbitrary, if $H$ is a quotient of $G$, then
\[
\beta(G) \geq \beta(H).
\]
\end{lemma}
\begin{proof}
Let $S'$ be a product-free subset of $H$ of size $\alpha(H)$.
The preimage of $S'$ in $G$
is product-free of size $\#S' \#G/\#H$, so
$\alpha(G) \geq \alpha(H)\#G/\#H$.
\end{proof}

In fact, one can prove an exact formula for $\alpha(G)$ showing
that this construction is essentially optimal. Many cases were
established around 1970, but only in 2005 was the proof of the 
following result finally
completed by Green and Ruzsa \cite{green-ruzsa}. 
\begin{theorem}[Green-Ruzsa]
Suppose that $G$ is abelian.
\begin{enumerate}
\item[(a)] If $n$ is divisible by a prime $p \equiv 2 \pmod{3}$,
then for the least such $p$, $\alpha(G) = \frac{n}{3} + \frac{n}{3p}$.
\item[(b)] Otherwise, if $3 | n$, then $\alpha(G) = \frac{n}{3}$.
\item[(c)] Otherwise, $\alpha(G) = \frac{n}{3} - \frac{n}{3m}$,
for $m$ the exponent (largest order of any element) of $G$.
\end{enumerate}
\end{theorem}
One possible explanation for the delay
is that it took this long for this subject to migrate into the mathematical
mainstream, as part of the modern subject of \textit{additive combinatorics}
\cite{tao-vu}; see Section~\ref{sec:interlude}.

The first appearance of the problem of computing $\alpha(G)$ for 
nonabelian $G$ seems to have been
in a 1985 paper of Babai and S\'os \cite{babai-sos}.
In fact, the problem appears there as an afterthought;
the authors were more interested in \textit{Sidon sets},
in which the equation $ab^{-1} = cd^{-1}$ 
has no solutions with $a,b,c,d$ taking
at least three distinct values. This construction can be related to 
embeddings of graphs as induced subgraphs of Cayley graphs;
product-free subsets arise because they relate to the special case
of embedding stars in Cayley graphs. 
Nonetheless, the Babai-S\'os paper is the first
to make a nontrivial assertion about $\alpha(G)$ for general $G$; see 
Theorem~\ref{T:babai-sos}.

This circumstance suggests rightly
 that the product-free problem is only one of a broad class of
problems about structured subsets of groups; this class can be
considered a nonabelian version of additive combinatorics,
and progress on problems in this class has been driven as much by
the development of the abelian theory as by interest from
applications in theoretical computer science. An example of the latter
is a problem of Cohn and Umans \cite{cohn-umans} (see also
\cite{cksu}): to find groups $G$ 
admitting large subsets $S_1, S_2, S_3$ such that the
equation $a_1 b_1^{-1} a_2 b_2^{-1} a_3 b_3^{-1} = e$, with
$a_i, b_i \in S_i$, has only solutions with $a_i = b_i$ for all $i$.
A sufficiently good construction would resolve an ancient problem
in computational algebra: to prove that two $n \times n$ matrices can be
multiplied using $O(n^{2+\epsilon})$ ring operations for any $\epsilon > 0$.

\section{Lower bounds: Duluth, 1994}

Upon my arrival at the REU in 1994, Joe gave me the paper of Babai and S\'os,
perhaps hoping I would have some new insight about Sidon sets. Instead,
I took the path less traveled and started thinking about 
product-free sets.

The construction of product-free subsets given in \cite{babai-sos} 
is quite simple: if $H$ is a proper subgroup of $G$, then any
nontrivial coset of $H$ is product-free. This is trivial to
prove directly, but it occurred to me to formulate it in terms of
permutation actions.  Recall that
specifying a transitive permutation action of the group $G$ is the same
as simply identifying a conjugacy class of subgroups: if $H$ is one
of the subgroups, the action is left multiplication on
left cosets of $H$. (Conversely, given an action, the point stabilizers
are conjugate subgroups.) The construction of Babai and S\'os can then
be described as follows.
\begin{theorem}[Babai-S\'os] \label{T:babai-sos}
For $G$ admitting a transitive action on $\{1,\dots,m\}$ with $m>1$,
$\beta(G) \geq m^{-1}$.
\end{theorem}
\begin{proof}
The set of all $g \in G$ such that $g(1) = 2$ is product-free of size $n/m$.
\end{proof}

I next wondered: what if you allow $g$ to carry 1 into a slightly
larger set, say a set $T$ of $k$ elements? 
You would still get a product-free set
if you forced each $x \in T$ to map to something not in $T$.
This led to the following argument.
\begin{theorem} \label{T:kedlaya}
For $G$ admitting a transitive action  on $\{1,\dots,m\}$ with $m>1$,
$\beta(G) \geq c m^{-1/2}$. 
\end{theorem}
\begin{proof}
For a given $k$, we compute a lower bound for the average size of
\[
S = \bigcup_{x \in T} \{g \in G: g(1) = x\}
- \bigcup_{y \in T} \{g \in G: g(1), g(y) \in T\}
\]
for $T$ running over $k$-element subsets of $\{2, \dots, m\}$.
Each set in the first union contains $n/m$ elements, and they are all disjoint,
so the first union contains $kn/m$ elements. 
To compute the average of a set in the second union, note that for fixed
$g \in G$ and $y \in \{2,\dots,m\}$, 
a $k$-element subset $T$ of $\{1, \dots, m\}$
contains $g(1), y, g(y)$ with probability $\frac{k(k-1)}{m(m-1)}$
if two of the three coincide and $\frac{k(k-1)(k-2)}{m(m-1)(m-2)}$
otherwise. A bit of arithmetic then shows that the average size of $S$
is at least
\[
\frac{kn}{m} - \frac{k^3n}{(m-2)^2}.
\]
Taking $k \sim (m/3)^{1/2}$, we obtain
$\alpha(G) \geq c n/m^{1/2}$. (For any fixed $\epsilon > 0$,
the implied constant can be improved to
$e^{-1} - \epsilon$ for $m$ sufficiently large;
see the proof of Theorem~\ref{T:lower bound}. On the other hand, the
proof as given can be made constructive in case $G$ is doubly transitive,
as then there is no need to average over $T$.)
\end{proof}

This gives a lower bound depending on the parameter $m$, which we can
view as the index of the largest proper subgroup of $G$. To state a bound
depending only on $n$, one needs to know something about the dependence
of $m$ on $n$; by 
Lemma~\ref{L:quotient}, it suffices to prove a lower bound on $m$ in terms
of $n$ for all \emph{simple} nonabelian
groups. I knew this could be done in
principle using the classification of finite
simple groups (CFSG); after some asking around,
I got hold of a manuscript by Liebeck and Shalev
\cite{liebeck-shalev} that included the bound I wanted, leading to
the following result from \cite{kedlaya}.
\begin{theorem}
Under CFSG, the group $G$ admits a transitive action on a set of size $1 < 
m \leq c n^{3/7}$.
Consequently, Theorem~\ref{T:babai-sos} implies
$\alpha(G) \geq cn^{4/7}$, whereas
Theorem~\ref{T:kedlaya} implies
$\alpha(G) \geq c n^{11/14}$.
\end{theorem}

At this point, I was pretty excited to have discovered something interesting
and probably publishable. On the other hand,
I was completely out of ideas! I had no hope of getting any stronger
results, even for specific classes of groups, and it seemed impossible
to derive any nontrivial upper bounds at all. In fact, Babai and S\'os
suggested in their paper that maybe $\beta(G) \geq c$ for all $G$;
I was dubious about this, but I couldn't convince myself that one couldn't
have $\beta(G) \geq c n^{-\epsilon}$ for all $\epsilon > 0$.

So I decided to write this result up by itself, as my first Duluth
paper, and ask Joe for another problem (which naturally he provided).
My paper ended up
appearing as \cite{kedlaya}; I revisited the topic when I was asked to submit
a paper in connection with being named a runner-up for the Morgan Prize
for undergraduate research, the result being \cite{kedlaya-amm}.

I then put this problem in a mental deep freezer,
figuring (hoping?) that my youthful foray into combinatorics 
would be ultimately forgotten, once I had made some headway with some
more serious mathematics, like algebraic number theory or algebraic
geometry. I was reassured by the
expectation that the nonabelian product-free problem was both intractable
and of no interest to anyone, certainly not to any serious mathematician.

Ten years passed.\footnote{If you do not recognize this reference, you may
not have
 read the excellent novel \textit{The Grasshopper King}, by fellow
Duluth REU alumnus Jordan Ellenberg.}

\section{Interlude: back to the future}
\label{sec:interlude}

Up until several weeks before the Duluth conference, I had been planning to
speak about the latest and greatest in algebraic number
theory (the proof of Serre's conjecture linking modular forms
and mod $p$ Galois representations, recently completed by
Khare and Wintenberger). Then I got an email that suggested that
maybe I should try embracing my past instead of running from it.

A number theorist friend  (Michael Schein) reported having attended
an algebra seminar  at Hebrew University about product-free
subsets of finite groups, and hearing my name in this
context. My immediate reaction was to wonder what self-respecting mathematician
could possibly be interested in my work on this problem.
The answer was Tim Gowers, who had recently established a nontrivial
upper bound for $\alpha(G)$ using a remarkably simple argument.

It seems that in the 
ten years since I had moved on to ostensibly more mainstream
mathematics, additive combinatorics had come into its own, thanks partly
to the efforts of no fewer than three Fields medalists (Tim Gowers, Jean
Bourgain, and Terry Tao); some sources date the start of this boom to
Ruzsa's publication in 1994 of a simplified proof \cite{ruzsa} of a theorem of
Freiman on subsets of $\ZZ/p\ZZ$ having few pairwise sums. 
In the process, some interest had spilled over
to nonabelian problems.

The introduction to Gowers's paper \cite{gowers}
cites\footnote{Since Joe is fond of noting ``program firsts'', I 
should point out that 
this appears to be the first citation of a Duluth paper
by a Fields medalist. To my chagrin, I think it
is also the first such citation of any of my papers.}
 my Duluth paper as giving the
best known lower bound on $\alpha(G)$ for general $G$. 
At this point, it became
clear that I had to abandon my previous plan for the conference in favor
of a return visit to my mathematical roots. 

\section{Upper bounds: bipartite Cayley graphs}

In this section, I'll proceed quickly through
Gowers's upper bound construction.
Gowers's paper
\cite{gowers} is exquisitely detailed; 
I'll take that fact as license to be
slightly less meticulous here.

The strategy of Gowers is to consider three sets $A,B,C$
for which there is no true equation
$ab=c$ with $a \in A, b \in B, c \in C$, and give an upper bound on
$\#A \#B \#C$.
To do this, he studies a certain \emph{bipartite Cayley graph}
associated to $G$. Consider the bipartite graph 
$\Gamma$ with vertex set $V_1 \cup V_2$, where each $V_i$
is a copy of $G$, with an edge from $x \in V_1$ to $y \in V_2$
if and only if $yx^{-1} \in A$.
We are then given that there are no edges
between $B \subseteq V_1$ and $C \subseteq V_2$.

A good reflex at this point would be to consider 
the eigenvalues of the adjacency matrix of $\Gamma$. For bipartite graphs,
it is more convenient to do something slightly different using singular values;
although this variant of spectral analysis of graphs is quite natural, I am only
aware of the reference \cite{bollobas-nikiforov} from 2004 (and only
thanks to Gowers for pointing it out).
Let $N$ be the \emph{incidence matrix}, with columns indexed by $V_1$
and rows by $V_2$, with an entry in row $x$ and column $y$ if $xy$
is an edge of $\Gamma$. 
\begin{theorem} \label{T:svd}
We can factor $N$ as a product
$U \Sigma V$ of $\#G \times \#G$ matrices over $\RR$, with $U,V$ orthogonal
and $\Sigma$ diagonal with nonnegative entries.
(This is called a \emph{singular value decomposition} of $N$.)
\end{theorem}
\begin{proof}
(Compare \cite[Theorem~2.6]{gowers}, or see any textbook on numerical
linear algebra.)
By compactness of the unit ball, there is a greatest 
$\lambda$ such that $\|N\bv\| = \lambda \|\bv\|$ for some nonzero $\bv \in 
\RR^{V_1}$.
If $\bv \cdot \bw = 0$, then $f(t) = \|N(\bv + t \bw)\|^2$
has a local maximum at $t=0$, so
\[
0 = \frac{d}{dt} \|N(\bv + t \bw)\|^2 = 2 t (N\bv) \cdot (N\bw).
\]
Apply the same construction to the orthogonal complement of $\RR \bv$ 
in $\RR^{V_1}$.
Repeating, we obtain an orthonormal basis of $\RR^{V_1}$; 
the previous calculation shows that the image of this basis in $\RR^{V_2}$
is also orthogonal. Using these to construct $V,U$ yields the claim.
\end{proof}

The matrix $M = NN^T$ is symmetric, and has
several convenient properties.
\begin{enumerate}
\item[(a)]
The trace of $M$ equals the number of edges of $\Gamma$.
\item[(b)]
The eigenvalues of $M$ are the squares of the diagonal entries of $\Sigma$.
\item[(c)]
Since $\Gamma$ is regular of degree $\#A$ and connected, the
largest eigenvalue of $M$ is $\#A$,
achieved by the all-ones eigenvector $\bfone$.
\end{enumerate}

\begin{lemma} \label{L:subspace}
Let $\lambda$ be the second largest diagonal entry of $\Sigma$.
Then the set $W$ of $\bv \in \RR^{V_1}$ with
$\bv \cdot \bfone = 0$ and $\|N \bv\| = \lambda \|\bv\|$ is 
a nonzero subspace of $\RR^{V_1}$.
\end{lemma}
\begin{proof}
(Compare \cite[Lemma~2.7]{gowers}.)
From Theorem~\ref{T:svd}, we obtain an orthogonal basis $\bv_1, \dots, \bv_n$
of $\RR^{V_1}$, with $\bv_1 = \bfone$, such that $N\bv_1, \dots, N\bv_n$
are orthogonal, and $\|N\bv_1\|/\|\bv_1\|, \dots, \|N\bv_n\|/\|\bv_n\|$
are the diagonal entries of $\Sigma$; we may then identify $W$
as the span of the $\bv_i$ with $i > 1$ and $\|N\bv_i\| = \lambda \|\bv_i\|$.

Alternatively, one may note that $W$ is obviously closed under scalar
multiplication, then check that $W$ is closed under addition
as follows.
If $\bv_1, \bv_2 \in W$, then 
$\|N (\bv_1 \pm \bv_2)\| \leq \lambda \|\bv_1 \pm \bv_2\|$, but by
the parallelogram law
\begin{align*}
\|N\bv_1 + N\bv_2\|^2 +
\|N\bv_1 - N\bv_2\|^2 &= 2 \|N\bv_1\|^2 + 2 \|N\bv_2\|^2 \\
&= 2 \lambda^2 \|\bv_1\|^2 + 2 \lambda^2 \|\bv_2\|^2 \\
&= \lambda^2 \|\bv_1 + \bv_2\|^2 + \lambda^2 \|\bv_1 - \bv_2\|^2.
\end{align*}
Hence $\|N (\bv_1 \pm \bv_2)\| = \lambda \|\bv_1 \pm \bv_2\|$.
\end{proof}

Gowers's upper bound on $\alpha(G)$ involves the parameter $\delta$,
defined as the smallest dimension of a 
nontrivial  representation\footnote{One could just as well restrict
to real representations, which would increase $\delta$
by a factor of 2 in some cases. For instance,
if $G = \PSL_2(q)$ with $q \equiv 3 \pmod{4}$, this would give
$\delta = q-1$.}
of $G$.
For instance, if $G = \PSL_2(q)$ with $q odd$, then
then $\delta = (q-1)/2$.

\begin{lemma} \label{L:gowers}
If $\bv \in \RR^{V_1}$ satisfies $\bv \cdot \bfone = 0$, then
$\|N\bv\| \leq (n\#A/\delta)^{1/2} \|\bv\|$.
\end{lemma}
\begin{proof}
Take $\lambda, W$ as in Lemma~\ref{L:subspace}.
Let $G$ act on $V_1$ and $V_2$ by right multiplication; then $G$
also acts on $\Gamma$. In this manner, $W$ becomes a real representation of $G$
in which no nonzero vector is fixed. In particular, $\dim(W) \geq \delta$.

Now note that the number of edges of $M$, which is
$n\#A$, equals the trace of $M$, which is at least $\dim(W) \lambda^2 \geq
\delta \lambda^2$.
This gives $\lambda^2 \leq n\#A/\delta$, proving the claim.
\end{proof}

We are now ready to prove Gowers's theorem
\cite[Theorem~3.3]{gowers}.
\begin{theorem}[Gowers] \label{T:gowers}
If $A,B,C$ are subsets of $G$ such that there is no true
equation $ab=c$ with $a \in A, b \in B, c \in C$, then
$\#A \#B \#C \leq n^3/\delta$. Consequently, $\beta(G)
\leq \delta^{-1/3}$.
\end{theorem}
For example, if $G = \PSL_2(q)$ with $q$ odd, then $n \sim c q^3$, so
$\alpha(G) \leq c n^{8/9}$. On the lower bound side, $G$ admits subgroups of
index $m \sim c q$, so $\alpha(G) \geq c n^{5/6}$.

\begin{proof}
Write $\#A = rn, \#B = sn, \#C = tn$. Let $\bv$ be the characteristic
function of $B$ viewed as an element of $\RR^{V_1}$, and put
$\bw = \bv - s \bfone$. Then 
\begin{align*}
\bw \cdot \bfone &= 0 \\
\bw \cdot \bw &= (1-s)^2 \#B + s^2 (n - \#B) = s(1-s)n \leq sn,
\end{align*}
so by Lemma~\ref{L:gowers}, $\|N\bw\|^2 \leq r n^2sn/\delta$.

Since $ab = c$ has no solutions with $a \in A, b \in B, c \in C$,
each element of $C$ corresponds to a zero entry in 
$N\bv$. However, $N \bv = N \bw + r sn \bfone$, so each zero entry in 
$N \bv$ corresponds to an entry of $N \bw$ equal to $-rsn$. Therefore,
\[
(tn)(rsn)^2 \leq \|N\bw\|^2 \leq rsn^3/\delta,
\]
hence $rst \delta \leq 1$ as desired.
\end{proof}

As noted by Nikolov and Pyber \cite{nikolov-pyber}, the
extra strength in Gowers's theorem is useful for other applications in group
theory, largely via the following corollary.
\begin{cor}[Nikolov-Pyber]
If $A,B,C$ are subsets of $G$ such that $ABC \neq G$,
then
$\#A \#B \#C \leq n^3/\delta$. 
\end{cor}
\begin{proof}
Suppose that $\#A \#B \#C > n^3/\delta$.
Put $D = G \setminus AB$,
so that $\#D = n - \#(AB)$. 
By Theorem~\ref{T:gowers}, we have $\#A \#B \#D \leq n^3/\delta$,
so $\#C > \#D$. Then for any $g \in C$, the sets $AB$ and $gC^{-1}$
have total cardinality more than
$n$, so they must intersect. This yields $ABC = G$.
\end{proof}

Gowers indicates that his motivation for this argument was the
notion of a \emph{quasi-random graph} introduced by
Chung, Graham, and Wilson \cite{cgw}. They show that (in a suitable
quantitative sense) a graph
looks random in the sense of having the right number of short cycles
if and only if it also looks random from the spectral viewpoint, i.e.,
the second largest eigenvalue of its adjacency matrix is not too large.

\section{Coda}

As noted by Nikolov and Pyber \cite{nikolov-pyber},
using CFSG to get a strong quantitative version of Jordan's theorem on finite
linear groups, one can produce upper and lower bounds for $\alpha(G)$ that 
look similar. (Keep in mind that the index of a proper subgroup must be
at least $\delta+1$, since any permutation representation of degree $m$
contains a linear representation of dimension $m-1$.)

\begin{theorem}
Under CFSG, the group $G$ has a proper subgroup of index 
at most $c \delta^2$. Consequently, 
\[
c n/\delta \leq \alpha(G) \leq cn/\delta^{1/3}.
\]
\end{theorem}
Moreover, for many natural examples (e.g., $G = A_m$ or $G = \PSL_2(q)$),
$G$ has a proper subgroup of index at most $c\delta$, in which case one has
\[
c n/\delta^{1/2} \leq \alpha(G) \leq cn/\delta^{1/3}.
\]
Since the gap now appears quite small, one might ask about closing it.
However, one can adapt the argument of \cite{kedlaya}
to show that Gowers's argument alone will not suffice, at least for
families of groups with $m \leq c \delta$.
(Gowers proves some additional results about products taken more than
two at a time \cite[\S 5]{gowers}; 
I have not attempted to extend this construction to that setting.)
\begin{theorem} \label{T:lower bound}
Given $\epsilon > 0$,
for $G$ admitting a transitive action on $\{1,\dots,m\}$ for $m$
sufficiently large,
there exist $A,B,C \subseteq G$ with $(\#A)(\#B)(\#C) \geq 
(e^{-1}-\epsilon)n^3/m$,
such that the equation $ab=c$ has no solutions with $a \in A, b \in B,
c \in C$. Moreover, we can force $B=C$, $C=A$, or $A=B^{-1}$ if desired.
\end{theorem}
\begin{proof}
We first give a quick proof of the lower bound $cn^3/m$.
Let $U,V$ be subsets of $\{1,\dots,m\}$ of respective sizes $u,v$.
Put
\begin{align*}
A &= \{g \in G: g(U) \cap V = \emptyset\} \\
B &= \{g \in G: g(1) \in U\} \\
C &= \{g \in G: g(1) \in V\};
\end{align*}
then clearly the equation $ab=c$ has no solutions with $a \in A, b \in B,
c \in C$. On the other hand,
\[
\#A \geq n - u \frac{vn}{m}, \qquad
\#B = \frac{un}{m},
\qquad
\#C = \frac{vn}{m},
\]
and so
\[
(\#A)(\#B)(\#C) \geq \frac{n^3}{m} 
\left( \frac{uv}{m} \right) \left( 1 - \frac{uv}{m} \right).
\]
By taking $u,v = \lfloor \sqrt{m/2} \rfloor$, we obtain 
$(\#A)(\#B)(\#C) \geq cn^3/m$.

To optimize the constant, we must average over choices of $U,V$.
Take $u,v = \lfloor \sqrt{m} \rfloor$. By
inclusion-exclusion, for any positive integer $h$,
the average of $\#A$ is bounded below by
\[
\sum_{i=0}^{2h-1} (-1)^i n \frac{u(u-1)\cdots(u-i+1)v(v-1)\cdots(v-i+1)}{i! m(m-1)
\cdots (m-i+1)}.
\]
(The $i$-th term counts occurrences of $i$-element subsets in $g(U) \cap V$.
We find $\binom{v}{i}$ $i$-element sets inside $V$; on average,
each one occurs
inside $g(U)$ for $n \binom{u}{i} / \binom{m}{i}$ choices of $g$.)
Rewrite this as
\[
n \left( \sum_{i=0}^{2h-1} (-1)^i \frac{(m^{1/2})^i 
(m^{1/2})^i}{m^i i!} +
o(1) \right),
\]
where $o(1) \to 0$ as $m \to \infty$.
For any $\epsilon > 0$, we have
\[
(\#A)(\#B)(\#C) \geq n^3 \frac{m}{m^2} 
\left( e^{-1} - \epsilon \right)
\]
for $h$ sufficiently large, and 
$m$ sufficiently large depending on $h$. 
This gives the desired lower bound.

Finally, note that we may achieve $B=C$ by taking $U = V$. 
To achieve the other equalities,
note that if the triplet $A,B,C$ has the desired property, so do 
$B^{-1},A^{-1},C^{-1}$ and $C,B^{-1},A$.
\end{proof}

I have no idea whether one can sharpen
Theorem~\ref{T:gowers} under the hypothesis $A=B = C$ 
(or even just $A=B$). It might be enlightening to collect some
numerical evidence using examples generated by Theorem~\ref{T:kedlaya};
with Xuancheng Shao, we have done this for $\PSL_2(q)$ for $q \leq 19$.

I should also mention again that (as suggested in
\cite{kedlaya-amm}) one can also study product-free
subsets of compact topological groups, which are large for Haar measure. Some such 
study is implicit in \cite[\S 4]{gowers}, but we do not
know what explicit bounds come out.

\end{document}